\documentclass[10pt,a4paper]{article}
\usepackage[latin1]{inputenc}
\usepackage[T1]{fontenc}
\usepackage{amsmath,amscd}
\usepackage{graphicx}
\usepackage{amsthm}
\usepackage{amssymb}
\usepackage{bbm}
\usepackage{ae}
\usepackage[all]{xy}

\newtheorem{sats}{Theorem}[section]
\newtheorem{deef}[sats]{Definition}
\newtheorem{lem}[sats]{Lemma}

\newtheorem{prop}[sats]{Proposition}

\newcommand{\Z}{\mathbbm{Z}}
\newcommand{\N}{\mathbbm{N}}

\newcommand{\ellL}{\mathcal{L}}

\newcommand{\rd}{\mathrm{d}}

\newcommand{\I}{\mathfrak{I}}
\newcommand{\J}{\mathcal{J}}
\newcommand{\Ko}{\mathcal{K}}
\newcommand{\Bo}{\mathcal{B}}
\newcommand{\Co}{\mathcal{C}}

\newcommand{\T}{\mathbbm{T}}

\newcommand{\Eg}{\mathcal{E}}

\renewcommand{\epsilon}{\varepsilon}
\renewcommand{\phi}{\varphi}

\title{
Equivariant extensions of $*$-algebras}
\author{Magnus Goffeng }
\date{Department of Mathematical Sciences, Division of Mathematics\\
Chalmers university of Technology and University of Gothenburg}

\begin{document}
\maketitle

\begin{abstract}
A bivariant functor is defined on a category of $*$-algebras and a category of operator ideals, both with actions of a second countable group $G$, into the category of abelian monoids. The element of the bivariant functor will be $G$-equivariant extensions of a $*$-algebra by an operator ideal under a suitable equivalence relation. The functor is related with the ordinary $Ext$-functor for $C^*$-algebras defined by Brown-Douglas-Fillmore. Invertibility in this monoid is studied and characterized in terms of Toeplitz operators with abstract symbol. 
\end{abstract}

\section*{Introduction}

Extensions of $C^*$-algebras by stable $C^*$-algebras have been thoroughly studied (see \cite{brbl}, \cite{bdf}, \cite{jeto}, \cite{thom}) due to their close relation to Toeplitz operators and $KK$-theory (see \cite{jeto}, \cite{thom}). The starting point was the article \cite{bdf} where an abelian monoid $Ext(A)$ was associated to a $C^*$-algebra $A$. This monoid consists of extensions $0\to \Ko\to E\to A\to 0$ under a certain equivalence relation, here $\Ko$ denotes the ideal of compact operators. The construction can be generalized to a bivariant theory by replacing $\Ko$ with an arbitrary stable $C^*$-algebra $B$ and one obtains an abelian monoid $Ext(A,B)$. In \cite{thom} this construction was put into the equivariant setting although only the invertible elements of $Ext_G(A,B)$ were studied. We will study the full extension monoids.

As is shown in \cite{jeto}, and equivariantly in \cite{thom}, an odd Kasparov $A-B$-module gives an extension of $A$ by $B$ which induces an additive mapping $KK^1_G(A,B)\to Ext_G(A,B)$. It can be shown, as is done in \cite{thom} that this is a bijection to the group $Ext^{-1}_G(A,B)\subseteq Ext_G(A,B)$ of invertible elements. A more straightforward approach is the proof in \cite{jeto} using the Stinespring representation theorem. As a corollary of this proof, if $A$ is nuclear and separable the Choi-Effros lifting theorem implies that $Ext_G(A,B)$ is a group if $G$ is trivial. This is the main motivation of studying extension theory. 

The reason for leaving the category of $C^*$-algebras is that most cohomology theories behave badly on $C^*$-algebras and one needs to look at dense subalgebras, see more in \cite{bejohn}. For example, if we use cohomology and Atiyah-Singers index theorem to calculate the index of a Toeplitz operator this is easily done via an explicit integral in terms of the symbol and its derivatives if the symbol is smooth, see more in \cite{fedo}.

With this as motivation we will extend the $Ext_G$-functor to $*$-algebras which embed into separable $C^*$-algebras and actions which extend to $C^*$-automorphisms. In the first part of this paper we define suitable categories for the first and the second variable of the functor. Then, similarly to the setting with $C^*$-algebras, we will construct a bivariant functor $\mathcal{E}xt_G$ to the category of abelian monoids. In particular there is a natural transformation 
\[\Theta:\mathcal{E}xt_G\to Ext_G\] 
in the category of abelian monoids. An interesting question to study further is what types of elements are in the kernel of the $\Theta$-mapping and if there is some way to make $\Theta$ surjective? 

After that we will move on to study the invertible elements. A rather remarkable result is that the invertible elements are those extensions which arise from a $G$-equivariant algebraic $\mathcal{A}-\I$-Kasparov modules. As an example, we will study the case of extensions of the smooth functions on a compact manifold by the Schatten class operators, in this case the $\Theta$-mapping turns out to be a surjection. At the end of the paper we describe a certain type of elements in the kernel of the $\Theta$-mapping which we will call linear deformations. The linear deformations are analytic in their nature. We end the paper by giving an explicit example of a linear deformation of the ordinary Toeplitz operators on the Hardy space that produces another $\Eg xt$-class but is homotopic to the $\Eg xt$-class defined by the ordinary Toeplitz operators.

\section{Definitions and basic properties}

To begin with we will define the suitable categories. From here on, let $G$ be a second countable locally compact group. We will say that the group action $\alpha :G\to Aut(A)$ acts continuously on the $C^*$-algebra $A$ if $g\mapsto \alpha_g(a)$ is continuous for all $a\in A$. 

\begin{deef}
Let $C^*A_G$ denote the category with objects consisting of pairs $(\mathcal{A},A)$ where  $A$ is a separable $C^*$-algebra with a continuous $G$-action and $\mathcal{A}$ is a $G$-invariant dense $*$-subalgebra. A morphism in $C^*A_G$ between $(\mathcal{A},A)$ to $(\mathcal{A}',A')$ is a $G$-equivariant $*$-homomorphism $\phi:\mathcal{A}\to \mathcal{A}'$ bounded in $C^*$-norm. 
\end{deef}

As an abuse of notation we will denote an object $(\mathcal{A},A)$ in $C^*A_G$ by $\mathcal{A}$ and its latin character $A$ will denote the ambient $C^*$-algebra. Observe that a morphism in $C^*A_G$ is the restriction of an equivariant $*$-homomorphism $\bar{\phi}:A\to A'$ uniquely determined by $\phi$. This follows from that if $\phi:\mathcal{A}\to \mathcal{A}'$ is bounded in $C^*$-norm it extends to $\bar{\phi}:A\to A'$ and since $\phi$ is equivariant $\bar{\phi}$ will also be equivariant. Conversely, an equivariant $*$-homomorphism of $C^*$-algebras is always $C^*$-bounded. When a linear mapping $T:\mathcal{A}\to \mathcal{A}'$, not necessarily equivariant, between two objects is induced by a bounded mapping $\bar{T}:A\to A'$ we will say that $T$ is $C^*$-bounded.

For a $C^*$-algebra $B$ we will denote its multiplier $C^*$-algebra by $\mathcal{M}(B)$ and embed $B$ as an ideal in $\mathcal{M}(B)$. If $B$ has a $G$-action we will equip $\mathcal{M}(B)$ with the induced $G$-action. 

\begin{deef}
If $(\mathfrak{I},I)\in C^*A_G$ satisfies that the $C^*$-algebra $I$ is equivariantly stable, that is $I\otimes \Ko\cong I$ where $\Ko$ has trivial $G$-action, and $\I$ is an ideal in $\mathcal{M}(I)$ the algebra $\mathfrak{I}$ is called a $C^*$-stable $G$-ideal. Let $C^*SI_G$ denote the full subcategory of $C^*A_G$ consisting of $C^*$-stable $G$-ideals.
\end{deef}

We will call a morphism $\psi:\I\to \I'$ of $C^*$-stable $G$-ideals an embedding of $C^*$-stable $G$-ideals if $\psi:I\to I'$ is an isomorphism.

\begin{prop}
\label{mtstab}
For any $C^*$-stable $G$-ideal $\I$ there is an equivariant isomorphism $M_2\otimes I\cong I$ inducing an isomorphism $M_2\otimes \I\cong \I$. The isomorphism is given by the adjoint action of a $G$-invariant unitary operator $V=V_1\oplus V_2:I\oplus I\to I$ between Hilbert modules.
\end{prop}

Notice that $V$ being unitary is equivalent to $V_1, V_2\in \mathcal{M}(I)$ being isometries satisfying 
\[V_1V_1^*+V_2V_2^*=1.\]

\begin{proof}
It is sufficient to construct two $G$-invariant isometries $V_1,V_2\in \mathcal{M}(I)$ such that $V_1V_1^*+V_2V_2^*=1$. Then $V:=V_1\oplus V_2 $ is a $G$-invariant unitary. Thus $V$ will be an isomorphism of Hilbert modules so $Ad \;V:M_2\otimes I\to I$ is an isomorphism and since $\I$ is an ideal $Ad \;V$ induces a isomorphism $M_2\otimes \I\cong \I$.

Let $K$ denote a separable Hilbert space with trivial $G$-action. Choose a unitary $V':K\oplus K\to K$. Let $V_1', V_2'\in \Bo (K)$ be defined by $V'(x_1\oplus x_2):=V_1'x_1 + V_2'x_2$. We may take the isometries  $V_1$ and $V_2$ to be the image of  $V_1'$ and $V_2'$ under the equivariant, unital embedding 
\[\Bo(K)=\mathcal{M}(\Ko)\hookrightarrow \mathcal{M}(I\otimes \Ko)\cong \mathcal{M}(I).\]
\end{proof}

One important class of $C^*$-stable $G$-ideals is the class of symmetrically normed operator ideals such as the Schatten class ideals and the Dixmier ideals (see more in \cite{casu}) over a separable Hilbert space $H$ with a $G$-action. In order to get equivariant stability we need to stabilize the Hilbert space with another Hilbert space with trivial $G$-action. Let $H'$ denote a separable Hilbert space and define 
\[\ellL^p_H:=(\ellL^p(H\otimes H'), \Ko(H\otimes H'))\] 
and analogously for the Dixmier ideal $\ellL^{n+}_H$. The $G$-action on the algebras are the one induced from the $G$-action on $H$.  \\

The main study of this paper are equivariant extensions $0\to \I\to \mathcal{E} \xrightarrow{\phi} \mathcal{A}\to 0$ where $\I$ is a $C^*$-stable $G$-ideal and $\mathcal{A}\in C^*A_G$. In particular we are interested in when such extensions admit $C^*$-bounded splittings of Toeplitz type. 

Consider for example the $0$:th order pseudodifferential extension $\Psi^0(M)$ on a closed Riemannian manifold $M$. This extension is an extension of the smooth functions on the cotangent sphere $S^*M$ by the classical pseudodifferential operators of order $-1$ given by the short exact sequence 
\[0\to  \Psi^{-1}(M) \to \Psi^0(M)\to C^\infty (S^*M)\to 0.\]
The algebra $\Psi^{-1}(M)$ is not $C^*$-stable, but $\Psi^{-1}(M)$ is dense in $\ellL^p(L^2(M))$ for any $p>n$, so the pseudo-differential extension fits in our framework after some modifications. The pseudo-differential extension admits an explicit splitting $T:C^\infty(S^*M)\to \Psi^0(M)$ in terms of Fourier integral operators which is not $C^*$-bounded if $\dim M>1$. Read more about this in Chapter $18.6$ in \cite{horm}. In this setting however, the problem can be mended. In \cite{guille} a $C^*$-bounded splitting is constructed for real analytic manifolds $M$ in terms of Grauert tubes and Toeplitz operators.

We will abuse the notation somewhat by  referring both to the object $\mathcal{E}$ and the extension by $\mathcal{E}$. Observe that the definition implies that there exists a commutative diagram with equivariant, exact rows
\[
\begin{CD}
0  @>>> \I @>>>\mathcal{E} @>\phi>>  \mathcal{A}@>>>0 \\
@. @VVV  @VVV @ VVV@.\\
0  @>>>  I @>>>E @>\bar{\phi} >> A@>>>0 \\
\end{CD} \]
The $*$-homomorphism $\bar{\phi}:E\to A$ is the extension of $\phi$ to $E$. 

\begin{deef}
Two $G$-equivariant extensions $\mathcal{E}$ and $\mathcal{E}'$ of $\mathcal{A}$  by $\I$ are said to be isomorphic if there exists a morphism $\psi:\mathcal{E}\to \mathcal{E}'$ in $C^*A_G$ that fits into a commutative diagram
\begin{equation}
\label{isoext}
\begin{CD}
0  @>>> \I @>>>\mathcal{E} @>\phi>> \mathcal{A}@>>>0 \\
@. @|  @VV \psi V @|@.\\
0  @>>>  \I @>>>\mathcal{E}' @>\phi' >> \mathcal{A}@>>>0 \\
\end{CD} 
\end{equation}
Because of the five lemma, $\psi$ is an isomorphism.
\end{deef}

Choose a linear splitting $\tau:\mathcal{A}\to \mathcal{E}$ and identify $\I$ with an ideal in $\mathcal{E}$. The mapping $\tau$ being a splitting of an equivariant mapping $\mathcal{E}\to\mathcal{A}$ implies that 
\begin{equation}
\label{split}
\tau(ab)-\tau(a)\tau(b), \;\;\tau(a^*)-\tau(a)^*\in \I \quad \mbox{and}  
\end{equation}
\begin{equation}
\label{split2}
\tau(g.a)-g.\tau(a)\in \I \; \forall g\in G.
\end{equation}
Given a $C^*$-stable $G$-ideal $\I$ we define the $G$-$*$-algebra $\mathcal{C}_{\I} :=\mathcal{M}(I)/\I$ and denote by $q_{\I}:\mathcal{M}(I)\to \mathcal{C}_{\I}$ the canonical surjection.  By the equations \eqref{split} and \eqref{split2} the mapping $q_\I\tau:\mathcal{A}\to \mathcal{C}_\I$ is an equivariant $*$-homomorphism. We will call the mapping $\beta_{\mathcal{A}}:=q_\I\tau$ the Busby mapping for the extensions $\Eg$. A Busby mapping that is $C^*$-bounded after composing with $\mathcal{C}_\I\to \mathcal{M}(I)/I$ is called bounded. A Busby mapping which can be lifted to a $C^*$-bounded $G$-equivariant $*$-homomorphism of $\mathcal{A}$ is called trivial.\\

For an equivariant $*$-homomorphism $\beta:\mathcal{A}\to \mathcal{C}_\I$ we can define the $*$-algebra 
\[\mathcal{E}_\beta:=\{a\oplus x\in \mathcal{A}\oplus \mathcal{M}(I): \beta(a)=q_\I(x)\}.\] 
The $*$-algebra $\mathcal{E}_\beta$ is closed under the $G$-action on $\mathcal{A}\oplus \mathcal{M}(I)$ so it is a $G$-$*$-algebra. Denote the norm closure of $\mathcal{E}_\beta$ in $A\oplus \mathcal{M}(I)$ by $E_\beta$. We have an injection $\I\to \mathcal{E}_\beta$  and a surjection $\mathcal{E}_\beta\to \mathcal{A}$. The kernel of $\mathcal{E}_\beta\to \mathcal{A}$ is $\I$, so the sequence $0\to \I\to \mathcal{E}_\beta\to \mathcal{A}\to 0$ is exact and the arrows are equivariant. The $*$-algebra $\mathcal{E}_\beta$ is a well defined object in $C^*A_G$, because Theorem $2.1$ of \cite{thom} states that the induced $G$-action on $E_\beta$ is continuous provided it is continuous on $I$ and on $A$. 

\begin{prop}
\label{spbu}
The equivariant $*$-homomorphism $\beta:\mathcal{A}\to \mathcal{C}_\I$ determines the extension up to a isomorphism, i.e  if $\mathcal{E}$ has Busby mapping $\beta$, $\mathcal{E}$ is isomorphic to  $\mathcal{E}_\beta$.
\end{prop}

\begin{proof}
Suppose that $\beta$ is Busby mapping for $\Eg$. Define $\psi: \mathcal{E}\to \mathcal{E}_\beta$ as
\[\psi(x):=\phi(x)\oplus x.\]
Since $\phi$ is equivariant, so is $\psi$. This makes the diagram \eqref{isoext} commutative, thus $\psi$ is an isomorphism of $G$-equivariant extensions. 
\end{proof}

The most useful class of $G$-equivariant extensions are the ones arising from algebraic $\mathcal{A}-\I$-Kasparov modules. This is defined as an algebraic generalization of Kasparov modules for $C^*$-algebras, see more in \cite{jeto}. 

\begin{deef}
A $G$-equivariant algebraic $\mathcal{A}-\I$-Kasparov module is a $C^*$-bounded $G$-equivariant representation $\pi:\mathcal{A}\to \mathcal{M}(I)$ and an almost $G$-invariant symmetry $F\in \mathcal{M}(I)$ that is almost commuting with $\pi(\mathcal{A})$, that is:  
\[g.F-F \in \I\;\; \forall\;g\in G\quad\mbox{and}\quad [F,\pi(a)]\in \I\;\; \forall \;a\in \mathcal{A}.\]
\end{deef}

Since $F$ is a grading we can define the projection $P:=(F+1)/2$. The pair $(\pi,F)$ induces a $*$-homomorphism
\begin{equation}
\label{t}
\beta:\mathcal{A}\to \mathcal{C}_\I,\; a\mapsto q_\I(P\pi(a)P).
\end{equation}
The requirement $[F,\pi(a)]\in \I$ together with $g.F-F\in \I$ implies that $\beta$ is an equivariant $*$-homomorphism.\\

Let $B_G(\mathcal{A},\I)$ denote the set of bounded $G$-equivariant Busby mappings on $\mathcal{A}$. This is the correct set to study extensions in. By Proposition \ref{spbu} the set of $G$-equivariant Busby mappings is the same set as the set of isomorphism classes of $G$-equivariant extensions. But we need some useful notion of equivalence of extensions, or by the previous reasoning an equivalence relation on $B_G(\mathcal{A},\I)$. For an object  $\mathfrak{I}\in C^*SI_G$ we define the almost invariant weakly unitaries 
\[U^{aw}(\I):=q_\I^{-1}(\{v\in \mathcal{C}_\I: g.v=v, \;v^*v=vv^*=1\}).\]
Let the almost invariant unitaries be defined as $U^a(\I):=U^{aw}(\I)\cap U(\mathcal{M}(\I))$.

\begin{deef}
Strong equivalence on $B_G(\mathcal{A},\I)$ is the equivalence of Busby mappings by the adjoint $U^a(\I)$-action on $\mathcal{C}_\I$. Weak equivalence on  $B_G(\mathcal{A},\I)$ is that of the adjoint $U^{aw}(\I)$-action on $\mathcal{C}_\I$.

Let $E_G(\mathcal{A},\I)$ denote the set of strong equivalence classes of $B_G(\mathcal{A},\I)$ and let $E^w_G(\mathcal{A},\I)$ denote the set of weak equivalence classes. Similarly let $D_G(\mathcal{A},\I)$ denote the set of strong equivalence classes of trivial Busby mappings and let $D^w_G(\mathcal{A},\I)$ denote the set of weak equivalence classes of trivial Busby maps.
\end{deef}

The isomorphism $\lambda :M_2\otimes \Co _\I\to \Co _\I$ induced by $Ad \; V$ from Proposition \ref{mtstab} can be used to define the sum of two $G$-equivariant Busby mappings $\beta_1, \beta_2\in B_G(\mathcal{A},\I)$ as
\[\beta_1+\beta _2:=\lambda \circ(\beta _1\oplus \beta_2):\mathcal{A}\to \mathcal{C}_\I.\]

\begin{prop}
The binary operation $+$ on $B_G(\mathcal{A},\I)$ induces a well defined abelian semigroup structure on $E_G(\mathcal{A},\I)$ independent of the choice of the unitary $V=V_1\oplus V_2$. The set $D_G(\mathcal{A},\I)$ is a subsemigroup.  
\end{prop}

The proof of the above proposition is the same as the proof of Lemma $3.1$ in \cite{thom} where the semigroup of equivariant extensions of a $C^*$-algebra is constructed. Two $G$-equivariant Busby mappings $\beta_1, \beta_2\in B_G(\mathcal{A},\I)$ are said to  be stably equivalent if they differ by trivial Busby mappings. That is, if there exist $C^*$-bounded, $G$-equivariant $*$-homomorphisms $\pi_1,\pi_2:\mathcal{A}\to \mathcal{M}(I)$ such that 
\[\beta_1\oplus q_\I\pi_1 \equiv \beta _2 \oplus q_\I \pi_2 :\mathcal{A}\to M_2\otimes \mathcal{C}_\I.\]
Stable equivalence induces a well defined equivalence relation on $E_G(\mathcal{A},\I)$ and $E^w_G(\mathcal{A},\I)$.

\begin{deef}
We define $\mathcal{E}xt_G(\mathcal{A},\I)$ as the monoid of stable equivalence classes of $E_G(\mathcal{A},\I)$ and $\mathcal{E}xt^w_G(\mathcal{A},\I)$ as the monoid of stable equivalence classes of $E^w_G(\mathcal{A},\I)$. For $G=\{1\}$ we denote the $\mathcal{E}xt$-invariants by $\mathcal{E}xt(\mathcal{A},\I)$ and $\mathcal{E}xt^w(\mathcal{A},\I)$.
\end{deef}

The monoids  $\mathcal{E}xt_G(\mathcal{A},\I)$ and $\mathcal{E}xt^w_G(\mathcal{A},\I)$ coincide with the semigroup quotients $E_G(\mathcal{A},\I)/D_G(\mathcal{A},\I)$, respectively  $E^w_G(\mathcal{A},\I)/D^w_G(\mathcal{A},\I)$. It has a zero-element since the class of an element in $D_G(\mathcal{A},\I)$ is zero. 

If we are given a $G$-equivariant extension $\mathcal{E}$ of $\mathcal{A}$ we will denote the class in $\mathcal{E}xt_G(\mathcal{A},\I)$ of its $G$-equivariant Busby mapping $\beta$ by $[\mathcal{E}]$ or by $[\beta]$.

\begin{prop}
\label{cstaanta}
If  $\I=I$ there are isomorphisms  
\[\mathcal{E}xt_G^w(\mathcal{A},I)\cong \mathcal{E}xt_G(\mathcal{A},I) \cong \mathcal{E}xt_G(A,I)\equiv Ext_G(A,I)\cong Ext_G^w(A,I).\] 
\end{prop}

\begin{proof}
We will prove the existence of the first and the second isomorphism. The proof of the last isomorphism is a special case of the first isomorphism for $\mathcal{A}=A$.

To prove the existence of the first isomorphism it is sufficient to show that weakly equivalent $G$-equivariant Busby mappings are strongly equivalent up to stable equivalence. Assume that $\beta_1, \beta_2\in B_G(\mathcal{A},\I)$ are weakly equivalent via the almost invariant weakly unitary $U\in U^{aw}(\I)$. Then $\beta_1\oplus 0$ and $\beta_2\oplus 0$ are weakly equivalent via the almost invariant weakly unitary $U\oplus U^*$. But the operator $U\oplus U^*$ lifts to a unitary $\tilde{U}\in \mathcal{M}(M_2\otimes I)$ since $\mathcal{C}_\I$ is a $C^*$-algebra. In fact $\tilde{U}\in U^a(M_2\otimes\I)$ since $U$ is almost invariant. Thus $\beta_1\oplus 0$ and $\beta_2\oplus 0$ are strongly equivalent. For the proof that $U\oplus U^*$ lifts to a unitary, see Proposition $3.4.1$ in \cite{brbl}.

The second isomorphism is given by the mapping $\mathcal{E}xt_G(\mathcal{A},I)\to \mathcal{E}xt_G(A,I)$, $[\mathcal{E}]\mapsto [E]$. In terms of the $G$-equivariant Busby mapping $\beta$ the mapping is given by $[\beta]\mapsto [\bar{\beta}]$, since $\mathcal{A}$ is dense and $\beta$ is bounded by assumption this is a surjection and $\bar{\beta}$ determines $\beta$ uniquely.
\end{proof}

The constructions of $Ext_G$ and $Ext_G^w$ are the same as $\mathcal{E}xt_G$ and $\mathcal{E}xt_G^w$ but with $C^*$-algebras. These constructions can be found in \cite{bdf}, \cite{jeto} and \cite{thom}. Proposition \ref{cstaanta} is a mild generalization of Proposition $15.6.4$ in \cite{brbl}. The proof is the same although $\mathcal{A}$ does not need to be a $C^*$-algebra.

Since the two theories are very similar we will focus on $\mathcal{E}xt_G$. All results stated in this paper are easily verified to also hold for $\mathcal{E}xt_G^w$.

\section{Functoriality of $\mathcal{E}xt_G$}

In this section we will prove that $\mathcal{E}xt_G$ is a functor to the category $Mo^{ab}$ of abelian monoids. We define this category to have objects of abelian monoids and a morphism is an additive mapping $k:M_1\to M_2$ such that $k(0)=0$. We know how $\mathcal{E}xt_G$ acts on the objects of  $C^*A_G$ and $C^*SI_G$. What needs to be defined is the action of $\mathcal{E}xt_G$ on the morphisms. We begin by showing that $\mathcal{E}xt_G$ depends covariantly on $\I$.

Let $\psi:\I\to \I'$ be a morphism of $C^*$-stable $G$-ideals. By definition $\psi$ can be extended to an equivariant mapping $\mathcal{M}(I)\to \mathcal{M}(I')$ which induces an equivariant mapping $q_\psi:\mathcal{C}_{\I} \to \mathcal{C}_{\I'}$. Define $\psi_* :E_G(\mathcal{A},\I)\to E_G(\mathcal{A},\I')$ by $\psi_*[\beta]:=[q_\psi\circ\beta]$. Clearly, $\psi_*[\beta]$ is independent of the stable equivalence class of $[\beta]$. Hence $\psi$ induces a well defined mapping 
\[\psi_*:\mathcal{E}xt_G(\mathcal{A},\I)\to \mathcal{E}xt_G(\mathcal{A},\I').\]
Since $\psi_*$ acting on a trivial extension gives a trivial extension we have a homomorphism of monoids. \\

Let us move on to proving that $\mathcal{E}xt_G$ depends contravariantly on $\mathcal{A}$. Let $\phi:\mathcal{A}\to \mathcal{A}'$ be a morphism in $C^*A_G$. Take a $G$-equivariant Busby mapping $\beta$ of $\mathcal{A}'$. Then we can define a $G$-equivariant Busby mapping $\phi^*\beta:=\beta\circ\phi$ of $\mathcal{A}$. This clearly depends on neither strong equivalence class nor stable equivalence class of the $G$-equivariant Busby mapping. If $\beta$ is trivial it follows that $\phi^*\beta$ is trivial so we have a morphism of monoids 
\[\phi^*:\mathcal{E}xt_G(\mathcal{A}',\I)\to \mathcal{E}xt_G(\mathcal{A},\I).\]

We have now proved the following proposition.

\begin{prop}
The functor $\mathcal{E}xt_G:C^*A_G\times C^*SI_G \to Mo^{ab}$ is a well defined functor. It is covariant in $\I$ and contravariant in $\mathcal{A}$.
\end{prop}

As noted above, an extension $\mathcal{E}$ of the algebra $\mathcal{A}$ by $\I$ gives rise to an extension $E$ of  $A$ by $I$. This procedure defines a mapping $E_G(\mathcal{A},\I)\to E_G(A,I)$ which respects stable equivalences. 

Let $C^*_G$ denote the category of separable $C^*$-algebras with a continuous $G$-action and $SC^*_G$ the full subcategory of equivariantly stable objects in $C^*_G$. We can define an essentially surjective functor 
\[\Gamma_1 :C^*A_G\times C^*SI_G\to C^*_G\times SC^*_G,\] 
\[((\mathcal{A},A),(\I,I))\mapsto (A,I).\]
Its right adjoint is the full and faithful functor 
\[\Gamma_2 :C^*_G\times SC^*_G\to C^*A_G\times C^*SI_G\] 
\[(A,I)\mapsto ((A,A),(I,I)).\] 
Notice that $\Gamma_1\Gamma_2$ is the identity functor on $C^*_G\times SC^*_G$. Define the functor 
\[Ext_G:C^*_G\times SC^*_G\to Mo^{ab}\quad  \mbox{by} \quad Ext_G:=\mathcal{E}xt_G \circ \Gamma _2.\]
As noted above this definition coincides with the definition of the $Ext_G$-functor in \cite{bdf} and \cite{jeto}.

\begin{prop}
The mapping $\Theta$ defines a natural transformation 
\[\Theta :\mathcal{E}xt_G\to Ext_G\circ \Gamma_1.\]
\end{prop}

\begin{proof} 
The mapping $\Theta^\mathcal{A}_\I$ merely extends Busby mappings to the object's $C^*$-closure, so $\Theta^\mathcal{A}_\I$ commutes with composition of morphisms in $C^*A_G\times C^*SI_G$ since they are just equivariant $C^*$-bounded $*$-homomorphisms. Thus $\Theta$ is a natural transformation.
\end{proof}

\section{Invertible extensions}

Just as in the case of a $C^*$-algebra one can relate invertibility in the $\mathcal{E}xt_G$-monoid and properties of the splitting. In this section we will study invertibility in $\mathcal{E}xt_G$-monoid in terms of Toeplitz operators.

The main result to be obtained in this section tells us that there is a direct link between algebraic properties in the $\mathcal{E}xt_G$-monoid and analytical properties of the extension. But this tells us nothing about how to construct the inverse or give explicit expressions. We will study this in the case of $G$ being the trivial group and for extensions admitting a $C^*$-bounded, completely positive splitting. Then these explicit constructions are possible in an ideal $\J_\I\supseteq \I$ such that $\I$ is the linear span of $\{a^*a:a\in \J_\I\}$. In this setting an explicit inverse can be given in $\mathcal{E}xt(\mathcal{A},\J_\I)$.

\begin{deef}
A $G$-equivariant extension which admits a splitting of the form $a\mapsto P\pi(a)P$, for a $G$-equivariant algebraic $\mathcal{A}-\I$-Kasparov module $(\pi,F)$ and $P=(F+1)/2$,  is called a $G$-equivariant Toeplitz extension.
\end{deef}

We will sometimes identify the Toeplitz extension with the pair $(P,\pi)$.

\begin{sats}
\label{psuminv}
An extension $[\mathcal{E}]\in \mathcal{E}xt_G(\mathcal{A},\I)$ is invertible if and only if $[\mathcal{E}]$ can be represented by a $G$-equivariant Toeplitz extension.
\end{sats}

For equivariant extensions of $C^*$-algebras this statement is proved in \cite{thom} (Lemma $3.2$) and the case $G$ trivial is well studied in \cite{jeto} and \cite{brbl}. Our proof of Theorem \ref{psuminv} is based upon the same ideas adjusted to our setting.

\begin{lem}
\label{mansalemmaett}
Every strong equivalence class of an invertible $G$-equivariant extension is stably equivalent to a $G$-equivariant Toeplitz extension.
\end{lem}

\begin{proof}
Assume that $\mathcal{E}$ is a $G$-equivariant extension of $\mathcal{A}$ by $\I$ with equivariant Busby mapping $\beta_1:\mathcal{A}\to \Co_\I$ which is invertible in $\mathcal{E}xt_G(\mathcal{A},\I)$. By definition there is a mapping 
$\beta_2:\mathcal{A}\to \Co_\I$ and a $U\in U^a(M_2\otimes \I)$ such that 
\[U^*(\beta _1\oplus \beta_2)U :\mathcal{A}\to M_2\otimes \Co_\I\]
can be lifted to an equivariant $C^*$-bounded representation $\pi:\mathcal{A}\to M_2\otimes \mathcal{M}(I)$. Let $P\in M_2\otimes \mathcal{M}(I)$ denote the almost $G$-invariant projection $U^* \begin{pmatrix} 1&0\\0&0 \end{pmatrix}U$. Define
\[\beta'(a):=q_\I(P\pi(a)P), \quad \beta''(a):=q_\I((1-P)\pi(a)(1-P)).\]
For $a\in \mathcal{A}$, we have
\[\beta_1(a)=q_\I(UPU^*)(\beta _1(a)\oplus \beta_2(a))q_\I(UPU^*)=\]
\[=q_\I(U)q(P\pi(a)P)q_\I(U^*)=q_\I(U)\beta'(a)q_\I(U^*),\]
which implies that up to strong equivalence $\beta$ is the Busby mapping of the extension. By the same reasoning $\beta''$ is strongly equivalent $\beta_2$. 

Define $\tau'(a):=P\pi(a)P$ and $\tau''(a):=(1-P)\pi(a)(1-P)$. We express the representation $\pi':=Ad\; U^* \circ\pi$ as follows
\[\pi'(a)=\begin{pmatrix}
U\tau'(a)U^*& \pi_{12}(a)\\
\pi_{21}(a)& U\tau''(a)U^* \end{pmatrix},\]
Since $q_\I\pi'=\beta_1\oplus \beta_2$, it follows that $\pi_{12}(a),\pi_{21}(a) \in \I$. The calculation
\[[P,\pi(a)]=U^*\left[\begin{pmatrix} 1&0\\0&0 \end{pmatrix}, \pi'(a)\right]U=U^*\begin{pmatrix}0&\pi_{12}(a)\\
-\pi_{21}(a)& 0\end{pmatrix}U \in M_2\otimes \I,\]
is a consequence of that $M_2\otimes \I$ is an ideal in $M_2\otimes I$ and implies that $\tau$ defines a $G$-equivariant Toeplitz extension.
\end{proof}

\begin{proof}[Proof of Theorem \ref{psuminv}]
If $[\mathcal{E}]$ is invertible it is given by a Toeplitz extension by Lemma \ref{mansalemmaett}. Conversely assume that $\mathcal{E}$ is a $G$-equivariant Toeplitz extension $(\pi,P)$ of $\mathcal{A}$. We define $P':=1-P$, $P_2:=P\oplus P'$, $\tau(a):=P\pi(a)P$ and $\tau'(a):=P'\pi(a)P'$. Then the claim from which the theorem will follow is that the Busby mapping $q_\I\circ \tau'$ defines an inverse to $\mathcal{E}$. To prove this, we define the almost $G$-invariant symmetry
\[U:=\begin{pmatrix} P& P'\\
P'& P\end{pmatrix}.\]
This symmetry satisfies $UP_2U=1\oplus 0$. We make the observation that $(\pi\oplus\pi, P_2)$ and $(U\pi\oplus\pi U, P_2)$ defines the same extension because of Proposition \ref{spbu} and that the pair $(\pi,P)$ are $\I$-almost commuting.  Since 
\[\pi(a)\oplus 0= UP_2U(\pi(a)\oplus \pi(a))UP_2U\]
it follows that 
\[[q_\I\circ \tau]+[q_\I\circ \tau']=[q_\I\circ(P_2(\pi\oplus \pi) P_2)]= [q_\I\circ(UP_2U^2(\pi\oplus \pi) U^2P_2U)]=\]
\[=[q_\I\circ(UP_2U(\pi\oplus \pi) UP_2U)]=[q_\I\circ \pi\oplus 0]=0.\]
\end{proof}

Suppose that we are in the situation $G=\{e\}$. In this case we are able to calculate an inverse to extensions admitting positive splitting if we enlarge the ideal somewhat. This should be thought of as passing from $\ellL^n(H)$ to $\ellL^{2n}(H)$. First we need an abstract notion of this procedure.

\begin{prop} 
\label{sqrot}
Suppose that $\I$ is a $C^*$-stable $G$-ideal. The $*$-algebra
\[\J_\I:=l.s.\{x\in I:x^*x\in \I \quad \mbox{and}\quad xx^*\in \I\}.\]
defines a $C^*$-stable $G$-ideal $(\J_\I,I)\in C^*SI_G$. We will call $\J_\I$ the square root of $\I$. 
\end{prop} 

\begin{proof}
Define the two $*$-invariant subsets $\J_\I^+:=\{x\in I:x^*x\in \I \}$ and $\J_\I^-:=\{x\in I:xx^*\in \I\}$.  For $x\in \J_\I^+$ and $a\in \mathcal{M}(I)$, $(xa)^*xa \in \I$ so $xa\in \J_\I^+$. Since $\J_\I^+$ is $*$-invariant, $ax\in \J_\I^+$. Similarly, if $x\in \J_\I^+$ and $a\in \mathcal{M}(I)$ we have that $ax(ax)^*\in \I$ so $ax\in \J_\I^-$ and $xa\in \J_\I^-$. The $*$-algebra $\J_\I\equiv l.s.(\J_\I^+\cap \J_\I^-)$ so $\J_\I$ is an ideal in $\mathcal{M}(I)$. There is an embedding $\I\subseteq \J_\I$ because $\I$ is a $*$-algebra, so $\J_\I$ is dense in $I$. 
\end{proof}

\begin{sats}
\label{bodid}
Let  $\mathcal{E}$ be an extension of $\mathcal{A}$ by $\I$ admitting a $C^*$-bounded splitting $\kappa$ extending to a completely positive contraction $\kappa:A\to \mathcal{M}(I)$. If $i:\I\to \mathcal{J}_\I$ is the embedding of $\I$ into its square root, $i_*[q_\I\circ \kappa]$ is invertible in $\mathcal{E}xt(\mathcal{A},\mathcal{J}_\I)$.
\end{sats}

Before proving this we need to review the useful construction of the Stinespring representation. This is a standard method for operator algebras and was first introduced by Stinespring in \cite{stine}. 

\begin{sats}[Stinespring Representation Theorem]
\label{mansalemmatva}
Assume that $A$ is a separable $C^*$-algebra, $I$ is a stable $C^*$-algebra and that $\kappa:A\to \mathcal{M}(I)$ is a completely positive mapping such that $\|\kappa\|\leq 1$. Then there exists a $*$-homomorphism $\pi_\kappa:A\to M_2\otimes \mathcal{M}(I)$ of $A$ such that
\[\begin{pmatrix}\kappa (a)&0\\0&0 \end{pmatrix}=\begin{pmatrix} 1&0\\0&0 \end{pmatrix}\pi_\kappa(a)\begin{pmatrix} 1&0\\0&0 \end{pmatrix}.\]
\end{sats}

The $*$-homomorphism $\pi_\kappa$ is called a Stinespring representation of $\kappa$. For proof see \cite{jeto}.

\begin{lem}
\label{satsabove}
Assume that $\kappa:A\to \mathcal{M}(I)$ is a completely positive contraction. In the notation above
\[\{a\in A :\kappa(a^2)-\kappa(a)^2\in \I\}=\{a\in A:[P,\pi_\kappa(a)] \in\mathcal{J}_\I\},\]
where $P:=\begin{pmatrix} 1&0\\0&0 \end{pmatrix}$.
\end{lem}

\begin{proof}
We express the representation as follows
\[\pi(a)=\begin{pmatrix}
\kappa(a)& \pi_{12}(a)\\
\pi_{21}(a)& \pi_{22}(a) \end{pmatrix},\]
where $\pi_{12}(a)=P\pi(a)(1-P)$ and so on. This implies that $\pi_{12}(a)^*=\pi_{21}(a^*)$. Since $\pi$ is a representation
\begin{equation}
\label{ekv1}
\begin{pmatrix}
\kappa(ab)& *\\
*& * \end{pmatrix}=\pi(ab)=\pi(a)\pi(b)=\begin{pmatrix}
\kappa(a)\kappa(b)+\pi_{12}(a)\pi_{21}(b)& *\\
*& * \end{pmatrix}.
\end{equation}
So
\[\kappa(ab)-\kappa(a)\kappa(b)=\pi_{12}(a)\pi_{21}(b).\]

Thus $\kappa(a^2)-\kappa(a)^2\in \I$ if and only if $\pi_{12}(a)\pi_{21}(a)\in \I$.  After polarization we only need to show that this is equivalent to the statement $[P,\pi_\kappa(a)]\in \mathcal{J}_\I$ for self adjoint $a$. But
\[[P,\pi(a)]=\begin{pmatrix}0&\pi_{12}(a)\\
-\pi_{21}(a)& 0\end{pmatrix}\]
implies
\begin{equation}
\label{ekv2}
|[P,\pi(a)]|^2=-[P,\pi(a)]^2=\begin{pmatrix} \pi_{12}(a)\pi_{21}(a)&0 \\
0& \pi_{21}(a)\pi_{12}(a)\end{pmatrix} \in M_2\otimes \I
\end{equation}
It follows from \eqref{ekv2} that $\pi_{12}(a)\pi_{21}(a)\in \I$ if and only if $|[P,\pi_\kappa(a)]|^2\in \I$ if and only if $[P,\pi_\kappa(a)]\in \mathcal{J}_\I$. 
\end{proof}

This proves Theorem \ref{bodid} since this implies that $\kappa$ defines a Toeplitz extension of $\mathcal{A}$ by $\mathcal{J}_\I$ and by Theorem \ref{psuminv} the element $i_*[q_\I\circ\kappa]$ is invertible in $\mathcal{E}xt(\mathcal{A},\mathcal{J}_\I)$.\\

To see that the square root of a $C^*$-stable ideal is needed sometimes, consider the example of the Besov space $\mathcal{A}=\mathcal{B}^{1/p}_p$ on the circle $S^1$. This carries a representation $\pi:\mathcal{A}\to \Bo(L^2(S^1))$ by multiplication as functions. Let  $P$ be the Hardy projection. By \cite{pelle}, if $a\in L^\infty(S^1)$ it holds that $[P,\pi(a)]\in \ellL^p(L^2(S^1))$ if and only if $a\in \mathcal{A}$. Making a similar decomposition of $\pi$ as in the proof of Lemma \ref{satsabove} one can show that the completely positive mapping $\tau(a):=P\pi(a)P$ is a splitting of an extension of $\mathcal{A}$ by $\ellL^{p/2}$. Since $\mathcal{A}\equiv\{a\in L^\infty(S^1): [P,\pi(a)]\in \ellL^p(L^2(S^1)\}$ it follows that $[q_{\ellL^{p/2}}\circ \tau]\in \mathcal{E}xt(\mathcal{A},\ellL^{p/2})$ is not invertible by Theorem \ref{psuminv}. But if $i:\ellL^{p/2}\to \ellL^p$ denotes the inclusion mapping (which coincides with the mapping constructed in Proposition \ref{sqrot}) the element $i_*[q_{\ellL^{p/2}}\circ \tau]\in\mathcal{E}xt(\mathcal{A},\ellL^{p})$ is invertible by Theorem \ref{psuminv}.

\section{Example: Extensions of $C^\infty(M)$ by Schatten ideals}

Commutative $C^*$-algebras have many good properties such as nuclearity and concrete realizations in geometry. The geometric interpretations of extensions of commutative $C^*$-algebras over a manifold, such as Toeplitz operators and pseudodifferential operators, are motivating for extension theory and allows for very concrete smooth $*$-subalgebras to do calculations in. 

For example, the one dimensional case $M=\T$ can be handled in a fairly straightforward fashion by finding an invertible generator for $\mathcal{E}xt^{-1}(C^\infty (S^1),\ellL^p)$ for $p\geq 2$ precisely as is done for $C(S^1)$ in Chapter $7$ in \cite{doug}. To find a set of generators in the general setting will be difficult. But a more abstract approach together with a topological description of $K$-homology of  smooth manifolds shows that the $\Theta$-mapping in fact is a surjection for $\mathcal{A}=C^\infty(M)$ and $\I$ being a Schatten ideal or a Dixmier ideal.

For $p>n$ define $i^p:\ellL^{n+}\to \ellL^p$ to be the embedding of $C^*$-stable ideals induced by the embedding $\ellL^{n+}\to \ellL^p$ of operator ideals. 

\begin{sats}
\label{sist}
Let $p>n$. Assume that $M$ is a compact manifold of dimension $n$ and  $\mathcal{A}=C^\infty(M)$.  Then the mappings
\[\Theta_{\ellL^{n+}}^{\mathcal{A}}:\mathcal{E}xt(\mathcal{A},\ellL^{n+}) \to Ext(C(M),\Ko)=K_1(M) \quad \mbox{and}\]
\[\Theta_{\ellL^{p}}^{\mathcal{A}}:\mathcal{E}xt(\mathcal{A},\ellL^p)\to Ext(C(M),\Ko) \quad \mbox{are surjective.}\]
\end{sats}

\begin{proof}
Using the definition of topological $K$-homology, see \cite{baumdo}, one sees that a class in  $K_1^{top}(M)\cong K^1(C(M))\cong Ext(C(M),\Ko)$ can be represented as the Fredholm module associated to a $0$:th order pseudodifferential operator $F$ over $M$ and the representation $\pi$ being pointwise multiplication of functions on $L^2(M,E)$ for some vector bundle $E$. Since $F$ is of order $0$ the commutator $[F,\pi(a)]$ is of order $-1$ for $a\in \mathcal{A}$. Thus $[F,\pi(a)]\in \ellL^{n+}(L^2(M,E))$ so $(F,\pi)$ is an $\mathcal{A}-\ellL^{n+}$-Kasparov module. Therefore $\mathcal{E}xt(\mathcal{A},\ellL^{n+}) \to Ext(C(M),\Ko)$ is surjective. A similar argument to the above one implies that $\Theta_{\ellL^{p}}^{\mathcal{A}}:\mathcal{E}xt(\mathcal{A},\ellL^{p}) \to Ext(C(M),\Ko)$ is surjective.
\end{proof}

\section{Deformations of Toeplitz extensions}

To end this paper we will look at a certain part of the set $\Theta^{-1}[(P,\pi)]$ for a Toeplitz extension $(P,\pi)$. The part of $\Theta^{-1}[(P,\pi)]$ we will study are linear perturbations of the projection $P$.  We will give an example of a smooth family of this type of linear deformations which gives a family of extensions $(x_\epsilon)_{\epsilon\in(1/2p,2/p)}\subseteq \Eg xt(C^\infty(S^1),\ellL^p)$ such that the the endpoints are non-equivalent. This example shows that $\Eg xt$ is not a homotopy invariant but carries more analytic information than similar bivariant theories.

If $(P,\pi)$ defines an $\I$-summable Toeplitz extension we say that $x\in \Eg xt(\mathcal{A},\I)$ is a linear deformation of $(P,\pi)$ by $T\in PIP$ if $x$ can be represented by an extension with a splitting of the form $\tau_T:a\mapsto (P+T)\pi(a)(P+T)$. Observe that $T\in PIP\subseteq I$ implies that $\Theta(P,\pi)=\Theta(x)$. For $a,b\in \mathcal{A}$ we have that 
\begin{align*}
\tau_T(ab)&-\tau_T(a)\tau_T(b)=\\
=&\;(P+T)\pi(ab)(P+T)-(P+T)\pi(a)(P+T)^2\pi(b)(P+T)=\\
=&\;\pi(ab)(P+T)^2(P-(P+T)^2)+[P+T,\pi(ab)](P+T)+\\
&+(P+T)\pi(a)[\pi(b),(P+T)^2](P+T)+\\
&+[\pi(ab),(P+T)](P+T)^3,
\end{align*}
so a sufficient condition for the operator $T$ to define a linear deformation is that $T^*-T,T^2+2T\in \I$ and $[T,\pi(a)]\in \I$ for all $a\in \mathcal{A}$.

The main example of a linear deformation is when one considers different representatives of Toeplitz extensions via a pseudo-differential operator on a manifold. Assume that $D$ is a self-adjoint, elliptic pseudo-differential operator on a smooth, compact manifold $M$ without boundary and let us take $P$ as the spectral projection onto the positive spectrum of $D$. The operator $P$ is a pseudo-differential operator of order $0$ so $[P,a]\in \ellL^p(L^2(M))$ for any $a\in C^\infty(M)$ and any $p>n$. Therefore the linear mapping $\tau(a):= PaP$ defines an $\ellL^p$-summable Toeplitz extension of $C^\infty(M)$. Let us take one more self-adjoint, elliptic pseudo-differential operator $K$ of order $\epsilon>n/2p$ and consider the order $-\epsilon$ operator 
\[T=P(K(1+K^2)^{-1/2}-1)P.\] 
The operator $T$ satisfies the identity 
\[T^2+2T=(T+P)^2-P=-P(1+K^2)^{-1}P.\]
So the operator $T$ satisfies $T^2+2T\in \ellL^p$ since we choose $K$ to have order bigger than $n/2p$. While $T$ is of order $-\epsilon$, $[T,\pi(a)]\in \ellL^p(L^2(M))$ and $T$ is self-adjoint since $K$ is self-adjoint. Therefore the linear mapping $\tau_T(a):=(P+T)a(P+T)$ defines an extension which is a linear deformation of $\tau$. 

The model case of the above setting is $K=D$. In this case the operator $P+T$ is given by $PD(1+D^2)^{-1/2}P$. Up to a finite rank operator, we have that $P=\frac{1}{2}(D|D|^{-1}+1)$ where the compact operator $|D|^{-1}$ can be defined as the inverse of $\sqrt{D^*D}$ on the range of $D^*D$ and defined to be $0$ on the finite-dimensional space $\ker(D^*D)$. Define the order $0$ pseudo-differential operator
\[\tilde{P}_D:=\frac{1}{2}(D(1+D^2)^{-1/2}+1).\]
Since $t/|t|-t(1+t^2)^{-1/2}=\mathcal{O}(t^{-2})$ as $t\to \infty$ and the order of $D$ is larger than $n/2p$ we have that 
\[PD(1+D^2)^{-1/2}P-\tilde{P}_D\in \ellL^p(L^2(M)).\]
Therefore the linear deformation of $\tau$ by $P(D(1+D^2)^{-1/2}-1)P$ coincides in $\Eg xt(C^\infty(M),\ellL^p)$ with the extension defined by the linear mapping $a\mapsto \tilde{P}_Da\tilde{P}_D$.

In general, we can not say more of $T$ than $T\in \ellL^{n/\epsilon}$ since the pseudo-differential operator $K(1+K^2)^{-1/2}-1$ is of order $-\epsilon$. As a consequence, if $\epsilon<n/p$ one can not expect that the mappings $q_{\ellL^p}\circ \tau$ and $q_{\ellL^p}\circ \tau_T$ coincide. We will by an example show that the two mappings may even lie in different strong equivalence classes. 

\begin{lem}
\label{uniteqlem}
Let $P$ be the Hardy projection on $S^1$ and assume that $T\in \Ko(H^2(S^1))$ is defined as $Tz^k:=\lambda_kz^k$ for some positive sequence $(\lambda_k)_{k\in\N}$ converging to $0$. If $a\in C^\infty(S^1)$ is given by $a(z):=z$ then for any $p\geq 1$ and any unitary $U\in \Bo(H^2(S^1))$ we have that 
\[\|U^*PaPU-(P+T)a(P+T)\|_{\ellL^p(H^2(S^1))}\geq \|T\|_{\ellL^p(H^2(S^1))}.\]
\end{lem}

\begin{proof}
We will use the notation $e_k(z):=z^k$ for $k\geq 0$ and $f_k:=Ue_k$. Our first observation is that 
\begin{equation}
\label{prpcalc}
(P+T)a(P+T)e_k=(1+\lambda_{k+1}+\lambda_k+\lambda_k\lambda_{k+1})e_{k+1}.
\end{equation}
If we set $L=U^*PaPU-(P+T)a(P+T)$ we have that 
\[L^*L=S_1+S_2-S_3-S_4\]
where 
\begin{align*}
S_1&:=U^*Pa^*PaPU, \\ 
S_2&:=(P+T)a^*(P+T)^2a(P+T),\\
S_3&:=(P+T)a^*(P+T)U^*PaPU \;\;\;\mbox{and}\\ 
S_4&:=U^*Pa^*PU(P+T)a(P+T).
\end{align*}
Using \eqref{prpcalc} we obtaing the following equalities:
\begin{align*}
\langle S_1e_k,e_k\rangle=\|Paf_k\|^2&=1,\\
\langle S_2e_k,e_k\rangle=\|(P+T)a(P+T)e_k\|^2&=(1+\lambda_{k+1}+\lambda_k+\lambda_k\lambda_{k+1})^2,\\
\langle S_3e_k,e_k\rangle=\overline{\langle S_3e_k,e_k\rangle}&=(1+\lambda_{k+1}+\lambda_k+\lambda_k\lambda_{k+1})\langle af_k,f_{k+1}\rangle.
\end{align*}
Using these calculations the fact that $\lambda_k,\lambda_{k+1}\geq 0$ together with the elementary estimate $|\langle af_k,f_{k+1}\rangle|\leq 1$ implies that
\begin{align*}
\langle L^*Le_k,e_k\rangle=&\;1+(1+\lambda_{k+1}+\lambda_k+\lambda_k\lambda_{k+1})^2-\\
&-2(1+\lambda_{k+1}+\lambda_k+\lambda_k\lambda_{k+1})\Re\langle af_k,f_{k+1}\rangle=\\
=&\;1-|\langle af_k,f_{k+1}\rangle|^2+\\
&+|1-\langle af_k,f_{k+1}\rangle+\lambda_{k+1}+\lambda_k+\lambda_k\lambda_{k+1}|^2\geq\\
\geq&\; (\lambda_{k+1}+\lambda_k+\lambda_k\lambda_{k+1})^2\geq |\lambda_k|^2.
\end{align*}
After reordering the sequence $\lambda_k$ into a decreasing sequence, we have that the singular values $(\mu_k(L))_{k\in \N}$ satisfies that $\mu_k(L)\geq \|Le_k\|\geq |\lambda_k|$, so by Lidskii's theorem
\[\|U^*PaPU-(P+T)a(P+T)\|_{\ellL^p(H^2(S^1))}^p=\sum_{k\in \N}\mu_k(L)^p\geq \sum_{k\in\N}  |\lambda_k|^p.\]
\end{proof}

\begin{prop}
For any $p>1$ there is a smooth family $(T_\epsilon)_{\epsilon\in(1/2p,2/p)}\subseteq \ellL^{2p}(H^2(S^1))$ such that the linear deformations of the Toeplitz extension on the Hardy space by $T_\epsilon$ defines a family $(x_\epsilon)_{\epsilon\in(1/2p,2/p)}\subseteq \Eg xt(C^\infty(S^1),\ellL^p)$ where $x_\epsilon\neq x_{\epsilon+1/p}$ for $\epsilon\in (1/2p,1/p)$.
\end{prop}

If we would replace the $\Eg xt$-invariant by for instance $kk$-theory, see more in \cite{cumero}, one would not be able to separate the elements $x_\epsilon$ and $x_{\epsilon+1/p}$ since the smooth family $(T_t)_{t\in[\epsilon,\epsilon+1/p]}$ can be used to construct a homotopy between the classification mappings of the extensions $x_\epsilon$ and $x_{\epsilon+1/p}$.

\begin{proof}
Let us start by defining the smooth family $(T_\epsilon)_{\epsilon\in(1/2p,2/p)}\subseteq \ellL^{2p}(H^2(S^1))$. We define $T_\epsilon$ for each $\epsilon\in (1/2p,2/p)$ in the same way as in Lemma \ref{uniteqlem} from the sequence
\[\lambda_{k,\epsilon}:=1-|k|^\epsilon(1+|k|^{2\epsilon})^{-1/2}.\] 
This choice of $\lambda_{k,\epsilon}$ coincides with that in the example above when $K=|\rd/\rd \theta|^{\epsilon}$. Since $\epsilon \mapsto \lambda_{k,\epsilon}$ is smooth, so is $\epsilon\mapsto T_\epsilon$. The sequence $(\lambda_{k,\epsilon})_{k\in\Z}$ behaves asymptotically as $|k|^{-\epsilon}$ so $(\lambda_{k,\epsilon})_{k\in\Z}\in \ell^{2p}(\N)$ since $\epsilon>1/2p$. 

When $\epsilon\in(1/p,2/p)$ the sequence $(\lambda_{k,\epsilon})_{k\in\Z}$ is $p$-summable. Therefore $(T_\epsilon)_{\epsilon\in(1/p,2/p)}\subseteq \ellL^p(H^2(S^1))$ and $\tau_{T_\epsilon}$ is isomorphic to the Toeplitz extension on the Hardy space for $\epsilon\in (1/p,2/p)$. However, when $\epsilon<1/p$ we have that $(\lambda_{k,\epsilon})_{k\in\Z}\notin \ell^{p}(\N)$. The norm estimate of the differences of the Toeplitz extension on the Hardy space and a deformation by $T_\epsilon$ in Lemma \ref{uniteqlem} implies that for any unitary $U\in \Bo(H^2(S^1))$ 
\[U^*PaPU-(P+T_\epsilon)a(P+T_\epsilon)\notin\ellL^p(H^2(S^1)).\]
Therefore $\tau$ is not strongly equivalent to $\tau_{T_\epsilon}$ for $\epsilon\in(1/2p,1/p)$ and $x_\epsilon\neq x_{\epsilon+1/p}$ for $\epsilon\in (1/2p,1/p)$.

\end{proof}

\end{document}